\RequirePackage{etex}
\documentclass[11pt]{article}
\usepackage{amsfonts}
\usepackage[T1]{fontenc}
\usepackage{calrsfs}
\usepackage{eco}
\usepackage{microtype}
\usepackage{typearea}
\usepackage{bigfoot}
\usepackage{amsmath}
\usepackage{amssymb}
\usepackage{amsthm}
\usepackage{yhmath}
\usepackage{braket}
\newcommand{\mynewdimen}{}  
\let\mynewdimen\newdimen    
\let\newdimen\newskip
\usepackage{easybmat}
\let\newdimen\mynewdimen
\usepackage{paralist}
\usepackage[british]{babel}

\theoremstyle{definition}
\newtheorem{theorem}{Theorem}
\newtheorem{corollary}[theorem]{Corollary}
\newtheorem{proposition}[theorem]{Proposition}
\numberwithin{theorem}{subsection}
\makeatletter
\def\fixequation{\let\c@equation\c@theorem\let\p@equation\p@theorem\let\theequation\thetheorem}
\def\@doit#1{\numberwithin{theorem}{#1}\fixequation}
\let\@tempstartsection\@startsection
\def\@startsection#1#2#3#4#5#6{\ifx#1\@empty\else\@doit{#1}\fi\@tempstartsection{#1}{#2}{#3}{#4}{#5}{#6}}
\makeatother
\theoremstyle{definition}
\newtheorem*{definition}{Definition}
\newtheorem{definition_}[theorem]{Definition}

\newtheorem*{def_thm}{Definition/Theorem}
\newtheorem{lemma}[theorem]{Lemma}

\newtheorem*{claim}{Claim}
\newtheorem{claim_}[theorem]{Claim}
\newtheorem*{example}{Example}
\newtheorem{example_}[theorem]{Example}
\newtheorem{problem}[theorem]{Problem}
\newtheorem*{remark}{Remark}
\newtheorem{remark_}[theorem]{Remark}
\newtheorem*{note}{Note}
\newtheorem*{notation}{Notation}

\newcommand{\calR}{\mathcal{R}}
\newcommand{\CC}{\mathbb{C}}
\newcommand{\PP}{\mathbb{P}}
\newcommand{\RR}{\mathbb{R}}
\newcommand{\g}{\mathfrak{g}}

\newcommand{\GL}{\operatorname{GL}}

\newcommand{\OO}{\mathcal{O}}
\newcommand{\ZZ}{\mathbb{Z}}

\newcommand{\defeq}{\mathrel{:=}}
\newcommand{\eqdef}{\mathrel{=:}}
\newcommand{\gives}{\leadsto}

\newcommand{\isom}{\cong}
\newcommand{\nullset}{\varnothing}
\newcommand{\placeholder}{\phantom{x}}
\newcommand{\tensor}{\otimes}
\newcommand{\transpose}[1]{#1\ensuremath{\sp{\text{T}}}}
\newcommand{\veps}{\varepsilon}

\newcommand{\Ad}{\operatorname{Ad}}

\newcommand{\Cauchy}{\operatorname{Cauchy}}
\newcommand{\diag}{\operatorname{diag}}
\newcommand{\End}{\operatorname{End}}

\newcommand{\ones}{\operatorname{\bf ones}}
\newcommand{\Pol}{\operatorname{Pol}}
\newcommand{\rank}{\operatorname{rank}}

\newcommand{\trace}{\operatorname{tr}}

\def\KandW{Kostant and Wallach}
\newcommand{\myhat}{\hat}
\newcommand{\mytilde}{\tilde}

{\catcode`\|=\active%
\gdef\bigset#1{\mathinner{\bigl\lbrace\,{\mathcode`\|"8000\let|\bigmid#1}\,\bigr\rbrace}}}
\def\bigmid{\egroup\bigm|\bgroup}
\newcommand{\bigsetminus}{\bigm\backslash}

\title{Linear algebra meets Lie algebra: the Kostant--Wallach theory}
\author{Noam Shomron \\ Beresford N. Parlett\footnote{Mathematics Department and Computer Science Division in the EECS Department, University of California, Berkeley, Calif., 94720 USA.  Email: \texttt{parlett@math.berkeley.edu}}}

\begin{document}
\bibliographystyle{alpha}

\setcounter{section}{-1}

\maketitle

\begin{abstract}
In two languages, Linear Algebra and Lie Algebra,
we describe the results of Kostant and Wallach on the fibre of matrices
with prescribed eigenvalues of all leading principal submatrices.
In addition,
we present a brief introduction to basic notions in Algebraic
Geometry, Integrable Systems, and Lie Algebra aimed at specialists in Linear Algebra.
\end{abstract}

\section{Introduction}

The well-known meeting-place of linear algebra and Lie algebra is the
classical matrix groups.  This paper is not a survey of that
theory, but is about a more specific confluence, and we consider only
the general linear (``type \textsc{a}'') case among the classical groups.

For want of a standard name we designate
as the Ritz values of a matrix the eigenvalues of all its leading
principal submatrices.  In \cite{KW1} and \cite{KW2}, Kostant and
Wallach studied the structure of the set (``fibre'') of matrices with
given Ritz values.  They constructed a certain commutative Lie group,
which acts on the space of matrices and whose action preserves Ritz
values; the second paper culminates by showing that this naturally
leads to a particularly nice set of coordinates on the space of
those matrices whose Ritz values satisfy some disjointness condition.

Inspired by this work, Parlett and Strang~\cite{PS} studied such
problems using bona fide matrix theory and linear algebra.  Later
one of us [BNP] showed quite explicitly how to parameterize the
space of matrices with given generic Ritz values, without invoking any
Lie theory or algebraic geometry.  However, hiding away the symmetry
of the problem does have some drawbacks: while the coordinates are
easy to define, it is not clear what they mean, or that they satisfy
any natural properties.  Thus the extra structure of Lie theory can
give depth to the matrix theory.

Kostant and Wallach's group, and their parameterization of generic
fibres, \emph{do} appear as such in the matrix-theoretic approach, but
the properties (such as their version of the Gelfand--Kirillov
theorem) that mark them as nice can only be seen by considering the
global geometry of the space of matrices, rather than a single fibre
at a time.

The purpose of the present work is twofold.  The primary purpose is
expository.  We need to introduce enough of the language of Lie theory
to be able to state and apply some of Kostant and Wallach's results.
We hope the reader will be convinced that geometrical intuition, using
the machinery of Lie/algebraic group theory and algebraic geometry,
while on the surface very abstract, can not only suggest the
\emph{right} way to think about a problem in linear algebra, but, in
fact, tell us how to do the actual computation.

We show how to recover BNP's construction using the language and
results of Kostant and Wallach, and prove that the two sets of
coordinates are, in fact, identical.

\section{Matrix picture}

\subsection{Notation and basic facts}
\label{sec:KW_theory}

In contrast to most papers on matrix theory,
certain matrices will be denoted by lower-case Roman letters,
such as~$x$, for compatibility with the notation used in~\cite{KW1,KW2}.
However, sometimes lower-case Roman letters, such as $b$~and~$c$,
denote vectors, and sometimes we will use them to denote scalars, such
as~$t$, etc.; the type of object
will always be unambiguous.
For a square matrix~$x$, the leading principal submatrix
of order~$m$ will be denoted by~$x_m$;
in \textsc{Matlab} notation $x_m=x(1:m,\,1:m)$.

Let~$E(x)$ denote the multiset of eigenvalues of~$x$.
The object of study is $\CC^{n\times n}$ for a fixed natural
number~$n$, but, since it will be endowed with extra (e.g., Lie)
structure, we use the standard notation~$M(n)$.
What is not standard is $\calR(x)$, \ $x\in M(n)$.
\begin{definition}
The set of \emph{Ritz values} of~$x\in M(n)$ is the
tuple~$\calR(x)=\bigl(E(x_1),E(x_2),\ldots,E(x_n)\bigr)$.
\end{definition}
This name was chosen because, in numerical linear algebra, when $x$ is
Hermitian $E(x_m)$, for~$m<n$, is regarded as an approximation to a
subset of~$E(x_n)$, and, independently, Rayleigh and Ritz showed that
the former are optimal approximations (in various senses) from the
subspace spanned by the first $m$~columns of the identity matrix.
See \cite[Chapter~11]{P}.

For Hermitian matrices there are interlacing conditions connecting
$E(x_{m-1})$ and $E(x_m)$.  However, for~$M(n)$, there are no
constraints on~$\calR(x)$; \emph{any} set of $\binom{n+1}{2}$ complex
numbers is $\calR(y)$ for some $y\in M(n)$.
Moreover,
sharing Ritz values determines an equivalence relation on~$M(n)$, and
we have
\begin{equation}
  \label{eq:fibres}
  M(n) = \coprod_{\mathcal{R}}M_{\calR}(n),
\end{equation}
where
$M_{\calR}(n) = \Set{ x\in M(n) | \calR(x) = \calR}$.
(The coproduct symbol~$\coprod$ here just means the set-theoretic disjoint union.)
In geometric theory
the equivalence class~$M_\calR(n)$ is called
a fibre,\footnote{Specifically, a fibre of the map~$x\mapsto\calR(x)$ which assigns
to each matrix its Ritz values.}
and we will use this terminology.
We have the following
\begin{problem}
\label{problem}
Given~$\calR$, describe the fibre~$M_\calR(n)$.
\end{problem}
The first observation is that $\calR$ determines the diagonal
entries uniquely:
\[
  x_{mm} = \sum E(x_m) - \sum E(x_{m-1}).
\]
Thus all members of~$M_\calR(n)$ share the same diagonal.

\paragraph*{Elementary conjugations.}
\makeatletter\@doit{subsection}\makeatother

For matrix theorists it seems a challenge to generate the fibre for a
given~$\calR$.  In the generic case all elements are similar to each
other, and yet the diagonal is fixed.  What mappings~$x\mapsto
gxg^{-1}$, with~$g\in\GL(n)$, preserve~$\calR$?

A little reflection suggests two types which we will call
\emph{elementary conjugations:}
\begin{compactenum}[(i)]
\item transposition (not conjugate transpose), $x\mapsto\transpose{x}$,
\item diagonal similarity, $x\mapsto dxd^{-1}$, \ $d\in\GL(n)$ diagonal.
\end{compactenum}
These two are far too weak to generate the fibre;
the effect of elementary conjugations upon the dual coordinates~$s$
which parameterize the fibre will be shown later.
\hfil\vadjust{\vskip\baselineskip}

To state the first significant result of
Kostant and Wallach, we need the notion of a Hessenberg matrix.
\begin{definition}
A matrix $H\in M(n)$ is upper Hessenberg if $H_{ij}=0$ for~$i>j+1$.
\ $H$ is said to be \emph{unreduced} if $H_{i+1,\ i}\neq 0$, \ $1\le
i\le n-1$. \ $H$ is said to be \emph{unit} upper Hessenberg if
$H_{i+1,\,i}=1$, \ $1\le i\le n-1$.
\end{definition}
The result of Kostant and Wallach is that
upper Hessenberg matrices serve as a natural set of representatives of each~$M_\calR(n)$.  Formally,
\begin{theorem}[\cite{KW1}, Theorem~0.1 and Remark~0.3]
\label{thm:upper_Hess}
For any~$\calR$, \ $M_\calR(n)$ contains exactly one unit upper
Hessenberg matrix.
\end{theorem}
\begin{remark}
A matrix-oriented proof was given in~\cite{PS}.
\end{remark}
Another result of Kostant and Wallach, also established by elementary
means in~\cite{PS}, is that
when $\calR$ is generic
(defined below in Definition~\ref{def:generic}) then the strictly lower
triangular part of~$x\in M_\calR(n)$ determines uniquely
the strictly upper triangular part, and vice versa.
Thus it is tempting to think of the strictly lower part
as a suitable set of coordinates for~$x$ that is
dual, or complementary, to~$\calR$.  The parameter
count~$\binom{n}{2}$ is exactly right.
For reasons that will be made clear below, this
temptation must be resisted.

The major result of~\cite{KW2} was to find a ``nice'' set of
coordinates to specify the members of~$M_\calR(n)$ for
generic~$\calR$.
They are given by tuples~$s=(s^{(1)},\ldots,s^{(n-1)})$, with
$s^{(m)}\in\CC^m$, but no entry can vanish, so we invoke~$\CC^\times$,
the multiplicative group~$\CC\setminus\{0\}$, and have
$s^{(m)}\in(\CC^\times)^m$.
Thus Kostant and Wallach present a coordinate system~$(\calR,s)$
for the generic elements of~$M(n)$ that is not familiar to matrix
theorists.
The goal of this paper is to show the geometric meaning of those
coordinates.  In some sense this is an instance of the Darboux
coordinates~$(q,p)$ in the Hamilton--Jacobi theory of
mechanics.\footnote{The coordinates~$(\calR,s)$ will be action-angle
  coordinates arising from an integrable system; see Section~\ref{sec:geometric_definition}.}

\subsection{Eigenvalue disjointness}
\label{sec:complementary_coordinates}

The simplest version of the theory, on which we will focus and for
reasons we will explain, occurs in
the ``generic'' case.
Consider the following conditions on an $n\times n$ matrix~$x$:
\begin{enumerate}
\def\labelenumi{\theenumi}
\def\theenumi{($\mathrm{G1}_m$)}\item The elements of~$E(x_m)$ are distinct.\label{G1m}
\def\theenumi{($\mathrm{G2}_m$)}\item $E(x_m)\cap E(x_{m+1})=\nullset$.\label{G2m}
\end{enumerate}
The significance of these conditions will be discussed in Section~\ref{sec:genericity}.
\begin{definition_}
\label{def:generic}
We call \ref{G1m} and \ref{G2m} the ``eigenvalue disjointness conditions.''
If both \ref{G1m}, $1\le m\le n$, and \ref{G2m}, $1\le m\le n-1$,
hold for~$x\in M(n)$, we will call $x$ \emph{generic.}
\end{definition_}
\begin{definition}
\[
  M_\Omega(n) = \Set{ x\in M(n) | x\text{\ is generic}}.
\]
\end{definition}
The complement of~$M_\Omega(n)$ in~$M(n)$ breaks up into pieces
specified by how badly conditions \ref{G1m}~and~\ref{G2m} are
violated.  Each such violation translates into the vanishing of some
polynomial in the entries of~$x$, e.g., ($\mathrm{G2}_1$)
is false exactly when~$x_{12}x_{21}=0$, and ($\mathrm{G1}_2$)
is false exactly
when~$(x_{11}-x_{22})^2+4x_{12}x_{21}=0$.
The set~$M_\Omega(n)$ is, therefore, a (nonempty, therefore dense) Zariski-open
subset of~$M(n)$.\footnote{If the appellation ``Zariski'' is
  intimidating, do not fret.  Our exposition eschews further mention
  of it.}
For this reason we sometimes say that a matrix~$x \in M_\Omega(n)$ has \emph{generic Ritz values.}
Often the term ``generic'' refers to any dense
open subset.  For example, one would say the condition that a matrix be
diagonalizable is a generic condition.  It is somewhat confusing to
refer to ``generic'' matrices, because there are many dense open
subsets.  In fact, since $M(n)\isom\CC^{n^2}$
is an algebraic variety, in the Zariski topology \emph{any} nonempty
open subset is dense.  In this paper, for the sake of brevity,
``generic'' will refer to the specific eigenvalue disjointness
conditions just described.

\paragraph*{What is wrong with eigenvalues?}
\makeatletter\@doit{subsection}\makeatother
A given set~$\calR$ of Ritz values may be designated
in various ways by a matrix theorist.
We could write down the eigenvalues~$E(x_1)$, \ $E(x_2)$, \ldots,
$E(x_n)$ in some specific order for each~$m$.
We could write down the set~$\{P_1,P_2,\ldots,P_n\}$ of monic
characteristic polynomials of~$x_1,\,x_2,\,\ldots,\,x_n$.
We could write down the coefficients of each~$P_m$ other than
the dominant one.  The descriptions are equivalent.
Life is not so carefree for Lie theorists because there is no
natural global meaning to ``the $i$th eigenvalue of~$x$.''
For example, let
\[
  x = x(t) = \begin{pmatrix} 0 & e^{2\pi it}\\1 & 0\end{pmatrix}.
\]
As $t$ goes from $0$ to $1$, \ $x(t)$ describes a smooth family of
generic matrices.  We may diagonalize $x(0)=\bigl(\begin{smallmatrix}0
  & 1\\ 1 & 0\end{smallmatrix}\bigr)$ via
\[
{\begin{pmatrix}
  1 & -1\\
  1 & 1
\end{pmatrix}}^{-1}
\begin{pmatrix}
  0 & 1\\
  1 & 0
\end{pmatrix}
\begin{pmatrix}
  1 & -1\\
  1 & 1
\end{pmatrix}
=
\begin{pmatrix}
  1 & 0\\
  0 & -1
\end{pmatrix}.
\]
Extending this to the family, we have
\[
{\begin{pmatrix}
  e^{\pi it} & -e^{\pi i t}\\
  1 & 1
\end{pmatrix}}^{-1}
x(t)
\begin{pmatrix}
  e^{\pi i t} & -e^{\pi it}\\
  1 & 1
\end{pmatrix}
=
\begin{pmatrix}
  e^{\pi i t} & 0\\
  0 & -e^{\pi it}
\end{pmatrix}
\eqdef \Lambda(t).
\]
However, $\Lambda(0)\neq\Lambda(1)$ despite the fact that $x(0)=x(1)$.
Hence there is no consistent smooth global way to order the
eigenvalues of a matrix.  But Lie algebra is committed to smooth maps.

In fact, Kostant and Wallach do give a global definition
of ``$i$th eigenvalue'' by means of a ``covering''~$M_\Omega(n,\mathfrak{e})$
of~$M_\Omega(n)$.
This extra technicality
is not needed for a description of the fibres (it is only introduced
in the second part~\cite{KW2} of their paper which establishes
a ``Gelfand--Kirillov theorem'' for~$M(n)$).
We avoid this complication by considering a single fibre~$M_\calR(n)$
with some ordering of each~$E(x_m)$ already given.

\subsection{The complementary coordinates}
\label{sec:the_complementary_coordinates}

Now we describe the complementary coordinates~$s = (s_1,\ldots,s_{\binom{n}{2}})$ for a generic fibre.
Consider a matrix~$x$ with generic Ritz values.
Write $\calR(x) = \bigl(E(x_1),\ldots,E(x_n)\bigr)$ with a fixed ordering for
each~$E(x_m)$.  We will denote by
\[
\Lambda_m=\diag(\mu^{(m)}_1,\ldots,\mu^{(m)}_m)
\]
the diagonal
matrix with the elements of~$E(x_m)$ placed along the diagonal.
For~$1\le m\le n-1$, \ref{G1m} implies that $x_m$ is similar
to~$\Lambda_m$.
Hence
there exists a matrix~$g_m\in\GL(m)$
such that
\begin{equation}
\label{eq:diagonalization_of_x_m}
x_m=g_m\Lambda_mg_m^{-1},
\end{equation}
and it becomes unique if the
last row of~$g_m$ consists of ones.
(Note that the last entry of an eigenvector of~$x_m$
must not vanish, since $(\lambda I_m-x_m)\bigl(\begin{smallmatrix}u\\0\end{smallmatrix}\bigr)=0$ implies $(\lambda I_{m-1}-x_{m-1})u=0$,
but $x_{m-1}$ and $x_m$ are assumed to have no eigenvalues in common.)
Then
\begin{equation}
\label{eq:arrow}
  x_{m+1} = \begin{pmatrix}g_m & 0 \\ 0 & 1\end{pmatrix}\begin{pmatrix}\Lambda_m & c_m\\ \transpose{b_m} & \delta_{m+1}\end{pmatrix}\begin{pmatrix}g_m^{-1} & 0\\0 & 1\end{pmatrix}
=
\begin{pmatrix}
  x_m & g_mc_m\\
  \transpose{b_m}g_m^{-1} & \delta_{m+1}
\end{pmatrix}
\end{equation}
(we consistently write rows as transposed columns),
and our dual coordinates appear in~\eqref{eq:arrow} as the entries
of~$\transpose{b_m}$.
We call the pair~$(\transpose{b_m},c_m)$ the ``arrow coordinates'' of~$x_{m+1}$.
\begin{claim_}
\label{claim:identical_coordinates}
$\transpose{b_m}$ is identical with \KandW's coordinates~$s^{(m)}$.
\end{claim_}
\noindent A proof is given in Section~\ref{sec:relation_to_arrow_coordinates}.

\begin{notation}
$\diag(v)$, $v\in\CC^m$, denotes the diagonal matrix~$\diag(v_1,\,v_2,\,\ldots,\,v_m)\in M(m)$.
\end{notation}
\noindent It is a consequence of the generic conditions \ref{G1m} and \ref{G2m}
that $\diag(b_m)\diag(c_m)$ is invertible:
\begin{theorem}
\label{thm:invertible}
$\diag(b_m)\diag(c_m)=-P_{m+1}(\Lambda_m){\bigl(P_m'(\Lambda_m)\bigr)}^{-1}$.
\end{theorem}
\begin{proof}
By~\eqref{eq:arrow},
\begin{equation}
\begin{split}
\label{eq:arrow_char_pol}
P_{m+1}(\lambda) &= \det(\lambda I_{m+1}-x_{m+1})
\\&= \det\begin{pmatrix}\lambda I_m-\Lambda_m & -c_m\\ -\transpose{b_m} & \lambda-\delta_{m+1}\end{pmatrix}
\\&= P_m(\lambda)\left[(\lambda-\delta_{m+1})-\transpose{b_m}(\lambda I_m-\Lambda_m)^{-1}c_m\right]
\\&= P_m(\lambda)(\lambda-\delta_{m+1})-\transpose{b_m}\diag\left(P_m^{\langle1\rangle}(\lambda),\ldots,P_m^{\langle m\rangle}(\lambda)\right)c_m,
\end{split}
\end{equation}
where
\begin{gather*}
  P_m^{\langle i\rangle}(\lambda)=\prod_{\mu\in E(x_m)\bigsetminus\{\mu^{(m)}_{i}\}}(\lambda-\mu), \qquad P_m^{\langle i\rangle}(\mu^{(m)}_{i})=P_m'(\mu^{(m)}_{i}).
\end{gather*}
Evaluating \eqref{eq:arrow_char_pol} at~$\lambda=\mu_1^{(m)},\,\mu_2^{(m)},\,\ldots,\,\mu_m^{(m)}$ gives
\[
  P_{m+1}(\Lambda_m) = 0-\diag(b_m)P_m'(\Lambda_m)\diag(c_m),
\]
and diagonal matrices commute.
\end{proof}
Given all~$\transpose{b_m}$, we
can reconstruct a unique~$x\in M_\calR(n)$.
First, a useful lemma:
\begin{lemma}
\label{lem:arrow}
Consider a (down) arrow matrix
\[
  A=\begin{pmatrix} D & p\\ \transpose{q} & \delta\end{pmatrix} \in M(m+1),
\]
where $D=\diag(d_i)$ and $A$ is similar to $\Lambda=\diag(\lambda_j)$,
and $d_i$~and~$\lambda_j$ are all distinct.
It is convenient to
define the (rectangular and skew) matrix~$\Cauchy(D,\Lambda)$,
\[
  \Cauchy(D,\Lambda)_{ij} = (d_i-\lambda_j)^{-1}
\]
(usually Cauchy matrices are defined as~$(d_i+\lambda_j)^{-1}$
with same-sized parameter sets).
Also, define~$\ones$ to be an array all of whose entries are $1$'s;
the shape of~$\ones$ is dictated by the context.
Then the spectral factorization of~$A$ is given by
\[
A=Z^{-1}\Lambda Z,
\]
where
\begin{gather*}
Z^{-1} = \begin{bmatrix}-\diag(p)\Cauchy(D,\Lambda)\\ \ones\end{bmatrix},
\quad
Z = \Pi^{-1}\left[\Cauchy(\Lambda,D)\diag(q),\,\ones\right],\\
\Pi=-\Cauchy(\Lambda,D)\diag(q)\diag(p)\Cauchy(D,\Lambda)+\ones,\\
\Pi_{ij}=0\text{\ if\ }i\neq j, \qquad
\Pi_{jj}=1+\sum_i\frac{p_iq_i}{(\lambda_j-d_i)^2}
=1-\sum_i\frac{\prod_m(d_i-\lambda_m)}{{(\lambda_j-d_i)}^2\prod_{k\neq i}(d_i-d_k)},\\
\text{(thus $\Pi$ is independent of $p$~and~$q$)}
\end{gather*}
i.e., we have found the eigenvectors of~$A$.
\begin{proof}
The distinctness of the $d_i$~and~$\lambda_j$
implies that the last
entry of an eigenvalue of~$A$ must be non-zero.
If $\mu$ is an eigenvalue of~$A$, the equations
\[
(\mu I-A)\begin{pmatrix}u\\1\end{pmatrix}=0, \qquad
\begin{pmatrix}\transpose{v} & 1\end{pmatrix}(\mu I-A)=0
\]
imply
\begin{equation}
  \label{eq:row_and_column_eigenvector}
  \begin{split}
    u &= {(\mu-D)}^{-1}p\\
    v &= {(\mu-D)}^{-1}q.
  \end{split}
\end{equation}
Using~\eqref{eq:row_and_column_eigenvector} for each eigenvalue~$\lambda_i$
of~$A$, we find
\begin{equation}
\label{eq:row_and_column_eigenvectors}
Z^{-1} = \overbrace{\begin{bmatrix}-\diag(p)\Cauchy(D,\Lambda)\\ \ones\end{bmatrix}}^{\text{\normalsize column eigenvectors}}, \qquad
\Pi Z = \overbrace{\begin{bmatrix}\Cauchy(\Lambda,D)\diag(q) & \ones\end{bmatrix}}^{\text{\normalsize row eigenvectors}}.
\end{equation}
The product of the matrices in~\eqref{eq:row_and_column_eigenvectors}
is not~$I$, but it must be diagonal since the $\lambda_j$ are simple
eigenvalues.\footnote{If $A\tilde{u}=\lambda_i\tilde{u}$ and
  $\transpose{\tilde{v}}A=\lambda_j\transpose{\tilde{v}}$ for some
  $\tilde{u},\,\tilde{v}\neq0$ with $\lambda_i\neq\lambda_j$, then $\transpose{\tilde{v}}\tilde{u}=0$.}
\[
\begin{split}
\Pi = \Pi ZZ^{-1} &= [\text{row eigenvectors}][\text{column eigenvectors}]\\
&= \begin{bmatrix}\Cauchy(\Lambda,D)\diag(q) & \ones\end{bmatrix}\begin{bmatrix}-\diag(p)\Cauchy(D,\Lambda)\\ \ones\end{bmatrix}\\
&= -\Cauchy(\Lambda,D)\diag(q)\diag(p)\Cauchy(D,\Lambda) + \ones.
\end{split}
\]
By Theorem~\ref{thm:invertible},
$\qquad q_ip_i=\displaystyle\prod_m(d_i-\lambda_m)\Bigg/\prod_{k\neq i}(d_i-d_k)$.
\end{proof}
\end{lemma}
Now we are ready to specify the dual coordinates.
Lemma~\ref{lem:arrow} in our case gives
\[
\begin{pmatrix}
  \Lambda_m & c_m\\
  \transpose{b_m} & \delta_{m+1}
\end{pmatrix}
=
\begin{pmatrix}
  -\diag(c_m)\Cauchy(\Lambda_m,\Lambda_{m+1})\\
  \ones
\end{pmatrix}
\Lambda_{m+1}
{\begin{pmatrix}
  -\diag(c_m)\Cauchy(\Lambda_m,\Lambda_{m+1})\\
  \ones
\end{pmatrix}}^{-1}.
\]
Substituting this in~\eqref{eq:arrow} gives
\[
  x_{m+1}=g_{m+1}\Lambda_{m+1}g_{m+1}^{-1}
\]
with
\[
g_{m+1}=
\begin{pmatrix}
  g_m & 0\\
  0 & 1
\end{pmatrix}
\begin{pmatrix}
  -\diag(c_m)\Cauchy(\Lambda_m,\Lambda_{m+1})\\
  \ones
\end{pmatrix}.
\]
Using Theorem~\ref{thm:invertible},
$\diag(c_m)=-P_{m+1}(\Lambda_m)P_m'(\Lambda_m)^{-1}\diag(b_m)^{-1}$,
and so we get the $g$-recurrence
\begin{align}
  g_1 &= (1),\nonumber\\
  g_{m+1} &=
  \begin{bmatrix}
    -g_m\diag(c_m)\Cauchy(\Lambda_m,\Lambda_{m+1})\\
    \ones
  \end{bmatrix} \label{eq:recurrence_with_c_m}\\
  &=
  \begin{bmatrix}
    g_mP_{m+1}(\Lambda_m){P_m'(\Lambda_m)}^{-1}{\diag(b_m)}^{-1}\Cauchy(\Lambda_m,\,\Lambda_{m+1})\label{eq:recurrence_with_b_m}\\
    \ones
  \end{bmatrix},
\end{align}
and we find $x=g_n\Lambda_ng_n^{-1}$.

The recurrence~\eqref{eq:recurrence_with_c_m} shows \emph{explicitly}
how $\Lambda_m$ and $\Lambda_{m+1}$, i.e., $E(x_m)$ and $E(x_{m+1})$,
determine the eigenvectors.  Clearly \eqref{eq:recurrence_with_c_m} is
simpler than \eqref{eq:recurrence_with_b_m}, but we give preference
to~$b_m$ over~$c_m$ to align our results with those of \KandW.

\newcommand{\Diag}[1]{D{\textstyle#1}}
We can put the parameters~$b_m$ together and define an invertible
diagonal $\binom{n}{2}\times\binom{n}{2}$ matrix
\[
  b = \diag(b_1,\ldots,b_{n-1})\in \Diag{\binom{n}{2}},
\]
where $\Diag{\binom{n}{2}}$ is the group of invertible diagonal
$\binom{n}{2}\times\binom{n}{2}$ matrices.
The recurrence~\eqref{eq:recurrence_with_b_m} constructs a matrix
$g_n=g^{(b)}\in\GL(n)$ given any $b\in\Diag{\binom{n}{2}}$, and we have
proved that
\[
M_\calR(n)\ni x \leftrightarrow \Set{ g_n\in\GL(n) |
  g_n\Lambda_ng_n^{-1}\in M_\calR(n)\text{ and }\transpose{e_n}g_n=\ones}
  \leftrightarrow b
\]
are bijections.
We can use these bijections to define an action of~$\Diag{\binom{n}{2}}$
on~$M_\calR(n)$ by
\begin{gather*}
  \Diag{\binom{n}{2}}\times M_\calR(n) \to M_\calR(n)\\
  b'\cdot x = g^{(bb')}\Lambda_n{(g^{(bb')})}^{-1}\mbox{\qquad if $x = g^{(b)}\Lambda_n{(g^{(b)})}^{-1}$.}
\end{gather*}
This is a description of \KandW's group action
in~\cite[Theorem~5.9]{KW2}, which is revisited in
Section~\ref{sec:geometric_definition} (see Remark~\ref{rem:generic_fibre_is_a_torus}) from a more natural point of view.

The conclusion (remember that we are still in the generic case) is
that each choice of
nonzero~$s=(\transpose{b_1},\,\transpose{b_2},\,\ldots,\,\transpose{b_{n-1}})$
will determine a member of~$M_\calR(n)$, and different $s$'s yield
different matrices in~$M_\calR(n)$.  For each fixed~$s$ as~$\calR$
ranges over all generic~$\calR$'s we get a transverse slice of~$M_\Omega(n)$.
It seems a blemish that our dual coordinates~$b_m$ had to be non-zero;
this will be removed naturally in the Lie format.
The canonical coordinates will be {\em angle
coordinates}~$q$, while the non-zero coordinates will essentially
appear as~$e^q$.  The non-zero coordinates do have the advantage of being
single-valued on the fibres, though.

\subsection{Complementary coordinates for the elementary conjugations.}

For completeness's sake, we describe the dual coordinates that
correspond to the elementary conjugations for generic~$\calR$.%

\begin{asparaitem}
\item[\bf Transposition.] Let
$\transpose{b}=(\transpose{b_1},\,\transpose{b_2},\,\ldots,\,\transpose{b_{n-1}})$
be the dual coordinates of~$x\in M_\calR(n)\subseteq M_\Omega(n)$.  To
find the dual coordinates of~$\transpose{x}$, it is necessary to
invoke the special diagonal matrices
\[
  \Sigma_m = -P_{m+1}(\Lambda_m)\left(P_m'(\Lambda_m)\right)^{-1}, \qquad 1\le m\le n-1,
\]
relating ${b_m}$ to $c_m$ which are given in
Theorem~\ref{thm:invertible}, and also the diagonal matrices~$\Pi_m$
(appearing as~$\Pi$ in Lemma~\ref{lem:arrow})
which relate row and column eigenvectors.
The diagonal matrices $\Sigma_m$~and~$\Pi_m$ depend only on~$\calR$,
not on~$b_m$.
\begin{lemma}
  Let the dual coordinates of~$\transpose{x}$ be
  $\transpose{\mytilde{b}}=(\transpose{\mytilde{b}_1},\,\transpose{\mytilde{b}_2},\,\ldots,\,\transpose{\mytilde{b}_{n-1}})$.
  Then, for~$1\le m\le n-1$,
\[
  \diag(\mytilde{b}_m) = \Pi_m\diag(b_m)^{-1}\Sigma_m.
\]
\end{lemma}
\begin{proof}
$x_m=g_m\Lambda_mg_m^{-1}$ implies
\[
  \transpose{x_m} = \transpose{(g_m^{-1})}\Lambda_m\transpose{g_m} =
  \transpose{(g_m^{-1})}\Pi_m\Lambda_m\Pi_m^{-1}\transpose{g_m},
\]
and the last row of~$\transpose{(g_m^{-1})}\Pi_m$ is $\ones$
by~\eqref{eq:row_and_column_eigenvectors}.
Transposition requires that the last row
of~$(\Pi_m^{-1}\transpose{g_m}\oplus1)\transpose{x_{m+1}}(\transpose{(g_m^{-1})}\Pi_m\oplus1)$
be equal to the last column of
\[
  (\Pi_mg_m^{-1}\oplus1)x_{m+1}(g_m\Pi_m^{-1}\oplus1)=
  \begin{pmatrix}
    \Lambda_m & \Pi_mc_m\\
    \transpose{b_m}\Pi_m^{-1} & \delta_{m+1}
  \end{pmatrix},
\]
therefore $\mytilde{b}_m = \Pi_mc_m$, and the Lemma follows
using $\diag(b_m)\diag(c_m)=\Sigma_m$.
\end{proof}
\item[\bf Diagonal similarity.] Let $\myhat{x}=dxd^{-1}$.  As usual,
  $d_m$ denotes the leading principal submatrix of~$d$, and so we
  denote the $(m,m)$th~entry by~$d(m)$.  Let
  $\transpose{\myhat{b}}=(\transpose{\myhat{b}_1},\,\transpose{\myhat{b}_2},\,\ldots,\,\transpose{\myhat{b}_{n-1}})$
  denote the dual coordinates
  of~$dxd^{-1}=\myhat{g}_n\Lambda_n\myhat{g}_n^{-1}$.  We note
  immediately that all $g$'s in~\eqref{eq:diagonalization_of_x_m} are
  normalized to have ones in the last row, so we cannot have
  $\myhat{g}_m=d_mg_m$.  To rectify this we define
  $\mathring{d}_m=d_m\big/d(m)$.  In particular,
\[
  d_mx_md_m^{-1}=\mathring{d}_mx_m\mathring{d}_m^{-1}.
\]
\begin{lemma}
For generic~$\calR$,
\[
  \myhat{b}_m = b_md(m+1)\big/d(m), \qquad 1\le m\le n-1.
\]
\end{lemma}
\begin{proof}
We have
\[
\myhat{x}_m = \mathring{d}_mx_m\mathring{d}_m^{-1} =
\mathring{d}_mg_m\Lambda_mg_m^{-1}\mathring{d}_m^{-1},
\]
so we calculate
\begin{multline*}
  (g_m^{-1}\mathring{d}_m^{-1}\oplus1)\myhat{x}_{m+1}(\mathring{d}_mg_m\oplus1)
  =
  (g_m^{-1}\mathring{d}_m^{-1}\oplus1)d_{m+1}x_{m+1}d_{m+1}^{-1}(\mathring{d}_mg_m\oplus1)
  =\\
  = d(m)\begin{pmatrix}g_m^{-1}\mathring{d}_m^{-1} & 0\\0 &
    1\end{pmatrix}\begin{pmatrix}\mathring{d}_m & 0\\0 & \frac{d(m+1)}{d(m)}\end{pmatrix}x_{m+1}\left[ d(m)\begin{pmatrix}g_m^{-1}\mathring{d}_m^{-1} & 0\\0 &
    1\end{pmatrix}\begin{pmatrix}\mathring{d}_m & 0\\0 &
    \frac{d(m+1)}{d(m)}\end{pmatrix}\right]^{-1}=\\
= \begin{pmatrix}d(m)I_m & 0 \\ 0 &
  d(m+1)\end{pmatrix}\begin{pmatrix}g_m^{-1} & 0 \\ 0 & 1\end{pmatrix}x_{m+1}\left[ \begin{pmatrix}d(m)I_m & 0 \\ 0 &
  d(m+1)\end{pmatrix}\begin{pmatrix}g_m^{-1} & 0 \\ 0 &
  1\end{pmatrix}\right]^{-1}=\\
=
\begin{pmatrix}
  d(m)I_m & 0\\
  0 & d(m+1)
\end{pmatrix}
\begin{pmatrix}
  \Lambda_m & c_m\\
  \transpose{b_m} & \delta_{m+1}
\end{pmatrix}
\begin{pmatrix}
  {d(m)}^{-1}I_m & 0\\
  0 & {d(m+1)}^{-1}
\end{pmatrix} =\\
=
\begin{pmatrix}
  \Lambda_m & \frac{d(m)}{d(m+1)}c_m \\
  \frac{d(m+1)}{d(m)}\transpose{b_m} & \delta_{m+1}
\end{pmatrix},
\end{multline*}
so that
\begin{align*}
  \myhat{c}_m &= c_m\bigl(d(m)\big/d(m+1)\bigr)
\intertext{and}
  \myhat{b}_m &= b_m\bigl(d(m+1)\big/d(m)\bigr),
\end{align*}
as claimed.
\end{proof}
\end{asparaitem}

One may ask what exactly is preserved by elementary conjugations, and
the answer is not Ritz values but rather all principal minor
determinants.  More precisely, a result due to
Loewy~\cite[Theorem~1]{Lo} is that under some non-degeneracy
conditions two matrices have equal corresponding principal minors if
and only if they are equivalent under an elementary conjugation.

There are $2^n-1$ non-trivial principal minors, but only $n^2-n+1$ of
them are independent, so the minors satisfy many relations, in
contrast to what happens for Ritz values.  While a full analysis of
the problem is outside the scope of this article, we emphasize that
this is again a prime example of the applicability of Lie- and
representation-theoretic and geometric methods to a problem in matrix
theory by exploiting its symmetry.

\subsection{Genericity conditions}
\label{sec:genericity}

Consider the following problem: let $B_m \in M(m)$ be any matrix,
and suppose we wish to find
\[
  B_{m+1} = \begin{pmatrix} B_m & c\\ \transpose{b} & \delta \end{pmatrix}\in M(m+1)
\]
such that
\begin{equation}
  \label{eq:char_pol_m+1}
  \det(\lambda I_{m+1}-B_{m+1}) = \prod_{1\le i\le m+1}(\lambda-\lambda_i)
\end{equation}
for given~$\lambda_1,\,\ldots,\,\lambda_{m+1}\in\CC$.
(Recall that~$\delta=\trace(B_{m+1})-\trace(B_m)$ is fixed.)
Equation~\eqref{eq:char_pol_m+1}
is an algebraic constraint on the $2m$ coordinates of $b$~and~$c$.
The $2m$ coordinates must satisfy $m$ polynomial equations, therefore
under sufficiently general conditions we expect an $m$-dimensional
set of solutions, while under degenerate conditions the dimension
may increase, or there may be no solutions at all.
Let us give several interpretations of this problem, and examine
the role of each genericity condition.
This will introduce useful notions from linear systems theory;
this is another field not generally known to matrix theorists.
These sections show why the strictly lower triangular part of~$x$ is
not a viable choice of coordinates for~$x$ complementary to~$\calR$.

\subsection{Observability and controllability}

\newcommand{\obspair}{\dbinom{B_m}{\transpose{b}}}
\newcommand{\contpair}{\Big(B_m\quad c\Big)}
Consider the system of ordinary differential equations given by
\begin{equation}
\begin{aligned}
& \dot x(t) = B_m x(t) + c u(t),\\
     & y(t) = \transpose{b}x(t) + \delta u(t).
\end{aligned}
\end{equation}
This represents a continuous time-invariant linear system
(SISO)\footnote{SISO $=$ Single Input, Single Output} with
state~$x(t)\in\CC^m$, control~$u(t)\in \CC$, and output~$y(t)\in\CC$.
Possibly abusing language,
\setbox0=\hbox{The pair~$\contpair$ is }\dimen0=\wd0
\setbox0=\hbox{ when $\rank(c,\,B_mc,\,\ldots)=m$.}
\advance\dimen0 by -\wd0
\[\hbox to\textwidth{\hfil\kern\dimen0\vbox{\halign{\hfil#\hfil\cr
\llap{The pair~$\obspair$ is }\emph{observable}\rlap{ when $\rank(b,\,\transpose{B_m}b,\,\ldots)=m$.}\cr
\noalign{\vskip1\jot}
\llap{The pair~$\contpair$ is }\emph{controllable}\rlap{ when $\rank(c,\,B_mc,\,\ldots)=m$.}\cr}}\hfil}\]

The algebraic significance of observability and controllability
will become clearer if we introduce the following terminology.
Let $\CC^m$ be the space of column vectors.
A vector $v\in\CC^m$ is called a \emph{cyclic} vector for a matrix~$B_m$
if for any $w\in\CC^m$ there exists a polynomial~$f(x)\in\CC[x]$
such that $f(B_m)v=w$.
In this language, the system~$\contpair$ is controllable
if and only if~$c$ is a cyclic vector for~$B_m$, and $\obspair$
is observable if and only if~$b$ is a cyclic vector for~$\transpose{B_m}$.
Matrix theorists would say that $c$ is a cyclic vector for~$B_m$
if the minimal polynomial of~$c$ for~$B_m$ has (maximal) degree~$m$:
$f(B_m)c=0$ for a nonzero polynomial~$f$ only if $\deg f\ge m$.
The \emph{centralizer} of an element~$x$ of a Lie algebra~$\g$ is the set of
elements~$\set{ y\in\g | xy=yx }$ that commute with~$x$.
(Lie algebras are briefly discussed elsewhere.  Here one may read $\g=M(n)$ and ignore the appellation.  Those familiar with functional analysis or operator algebras will also recognize this as the definition of the \emph{commutant} of~$\{x\}\subseteq M(n)=\End(\CC^n)$.)
\begin{theorem}
A matrix~$B_m$ has a cyclic vector if and only if
the centralizer of~$B_m$ coincides with the algebra~$\Set{ f(B_m) |
  f\in\CC[x]}$ of polynomials
in~$B_m$.
\end{theorem}
This property is known to Lie theorists as \emph{regularity}
(not to be confused with the property of being invertible---a matrix
with only zero eigenvalues may well be regular in our sense).
For matrix theorists,
it is equivalent to being non-derogatory, i.e., the minimal
polynomial equals the characteristic polynomial.
Clearly the identity element is far from regular; a diagonal matrix is
regular if and only if its diagonal entries are distinct.

Let
$P_m(\lambda)=\det(\lambda I_m-B_m)$ and
$P_{m+1}(\lambda)=\det(\lambda I_{m+1}-B_{m+1})$.
Block elimination yields
\begin{equation}
  \label{eq:schur_complement}
  P_{m+1}(\lambda) = P_m(\lambda)\bigl(\lambda-\delta-\transpose{b}(\lambda I_m-B_m)^{-1}c\bigr).
\end{equation}

Assume~\eqref{eq:char_pol_m+1}, that is, that $B_{m+1}$ has the
specified eigenvalues.
\begin{theorem}
Given $b$, there exists a unique~$c$ such that~\eqref{eq:char_pol_m+1}
is satisfied if and only if\/~$\dbinom{B_m}{\transpose{b}}$ is observable.
Given $c$, there exists a unique~$b$ such that~\eqref{eq:char_pol_m+1}
is satisfied if and only if\/~$\bigl(B_m\ c\bigr)$ is controllable.
\end{theorem}
\begin{proof}
This is a variation on a standard problem in control theory.
Re-write~\eqref{eq:schur_complement} as
\[
  \transpose{b}(\lambda I_m-B_m)^{-1}c = \lambda - \delta - \frac{P_{m+1}(\lambda)}{P_m(\lambda)}.
\]
By construction (recall that $\delta$ is determined by
$\delta=\trace(B_{m+1})-\trace(B_m)$), the right-hand side is
holomorphic at~$\lambda=\infty$,
and expanding both sides into power series gives
\[
   \sum_{k=1}^\infty\lambda^{-k}\transpose{b}B_m^{k-1}c = \sum_{k=1}^\infty\lambda^{-k}g_k
\]
for some numbers~$g_k\in\CC$.
Equating powers of~$\lambda$,
$b$ and $c$ must satisfy the (infinite) system of equations
\[
  \transpose{b}B_m^{k-1}c = g_k, \qquad k=1,2,\ldots
\]
The condition that there exist a unique solution~$c$ is exactly
that $\transpose{\bigl(\begin{matrix}\transpose{b} & \transpose{b}B_m & \cdots\end{matrix}\bigr)}$ have full column rank, i.e., observability.

The proof considering controllability is analogous.
\end{proof}
\begin{note}
Condition~\ref{G2m}
holds exactly when $P_{m+1}(\lambda)$ and $P_m(\lambda)$ are
relatively prime,
in which case the right-hand side of
\[
  \sum_{k=1}^\infty\lambda^{-k}\transpose{b}B_m^{k-1}c = \transpose{b}(\lambda I_m-B_m)^{-1}c = \lambda - \delta - \frac{P_{m+1}(\lambda)}{P_m(\lambda)}.
\]
is a rational function of degree~$m$.
In that case, if $b$ and $c$ constitute a solution,
the Hankel matrix
\[
\begin{pmatrix}
\transpose{b}c & \transpose{b}B_mc & \cdots\\
\transpose{b}B_mc & \transpose{b}B_m^2c & \cdots\\
\vdots
\end{pmatrix}
=
\begin{pmatrix}
\transpose{b}\\
\transpose{b}B_m\\
\transpose{b}B_m^2\\
\vdots & \vdots
\end{pmatrix}
\begin{pmatrix}
c & Bc & \cdots
\end{pmatrix}
\]
has rank~$m$ and so the system must be both
observable and controllable.

\end{note}

\begin{example}
If $B_m$ is an unreduced upper Hessenberg matrix, then the
row~$(0,\ldots,0,1)$ always yields an observable system.
\end{example}
\begin{example_}
\label{ex:observable_vectors}
Consider the case when~$B_m$ is a regular diagonal matrix
(thus~\ref{G1m} holds, and \ref{G2m} by assumption;
we are not assuming~($\mathrm{G1}_{m+1}$)).
Then \eqref{eq:char_pol_m+1} has solutions if and only if
$\obspair$ is observable,
if and only if $\contpair$ is controllable,
if and only if all the entries of~$b$ and of~$c$ are non-zero. 
(See Theorem~\ref{thm:invertible}.)
\end{example_}
If $B_m$ is regular and semi-simple, but not diagonal,
then it is still true that $\obspair$ is observable
if and only if~$\contpair$ is controllable:
if~$g_m^{-1}B_mg_m$ is diagonal, then this happens
when $\transpose{b}g_m$, resp.\ $g_m^{-1}c$, has non-zero entries.
(This is relevant to the matrix in~\eqref{eq:arrow}.)

A matrix~$B_m$ has a cyclic vector (if and only if~$\transpose{B_m}$
has a cyclic vector), if and only if~$B_m$ is regular.
So if~$B_m$ is not regular (in our sense), then the system~$\obspair$
is \emph{never} observable, nor is~$\contpair$ ever controllable.

\subsection{Beyond the generic case}

The criterion in Example~\ref{ex:observable_vectors},
that the entries of the row/column be non-zero,
may be readily generalized to the case when
$B_m$ is regular, but~\ref{G1m} fails to hold.
\begin{lemma}
Suppose that $B_m$ is in the
form
\[
  B_m = \diag\bigl(J_{m_1}(d_1),\ldots,J_{m_t}(d_t)\bigr),
\]
where each Jordan block
\[
  J_{m_i}(d_i) = \left(\begin{BMAT}(e){ccc}{cccc}d_i & 0 & \hdots\\
                               1  & d_i & \hdots\\
                               0  &  1 & \hdots\\
                               \vdots & &\\
               \end{BMAT}\right)
  \in M(m_i)
\]
and $d_1,\,\ldots,\,d_t$ are distinct.
Then a row~$(y_1,\ldots,y_m)$ is observable
if and only if each of the entries~$y_{m_1},\,y_{m_1+m_2},\,\ldots,\,y_m$
is nonzero, namely the last entry in each segment.
\end{lemma}
\begin{proof}
[Sketch of proof]
The matrix~$\transpose{B_m}$ is block-diagonal; let us write
\begin{gather*}
  \transpose{B_m} =
  \transpose{J_{m_1}(d_1)}\oplus\cdots\oplus\transpose{J_{m_t}(d_t)},\\
  \transpose{(y_1,\ldots,y_m)} = y^{(1)}\oplus\cdots\oplus y^{(t)}, \qquad
  y^{(i)} = \transpose{\left(y_{k_i+1},\ldots,y_{k_i+m_i}\right)}, \qquad k_i = \sum_{t'<i}m_{t'}
\end{gather*}
so that
\[
  \transpose{B_m}\transpose{(y_1,\ldots,y_m)} =
  \bigoplus_{i=1}^m\transpose{J_{m_i}(d_i)}y^{(i)}.
\]

We claim that $\transpose{(y_1,\ldots,y_m)}$ is a cyclic vector
for~$\transpose{B_m}$ if and only if $y^{(i)}$ is a cyclic vector
for~$\transpose{J_{m_i}(d_i)}$ for each~$1\le i\le t$.
One way to see this is to
recall that a vector~$y$ is cyclic for any
matrix~$B_m$ if and only if $B'y=0$ implies $B'=0$, for every $B'$ in
the centralizer of~$B_m$.
Because our $\transpose{B_m}$ is regular, its centralizer just consists of
block-diagonal matrices~$B'=J_{m_1}'\oplus\cdots\oplus J_{m_t}'$,
where each $J_{m_i}'$ commutes with~$\transpose{J_{m_i}(d_i)}$.
This verifies our claim.
This reduces the problem to
calculating cyclic vectors for a Jordan block, which is
straightforward.
\end{proof}
Hence, if some generalized eigenvalue of a regular~$B_m$ has multiplicity,
the set of rows making the system observable, and also
the set of columns making the system controllable, is
isomorphic to~$(\CC^\times)^t\times\CC^{m-t}$.  (And we recall that
neither observability nor controllability is possible when $B_m$ is
not regular.)

We conclude with the observation
that the geometric structure of the set of solutions
for $b$~and~$c$ satisfying \eqref{eq:char_pol_m+1},
at least for regular~$B_m$, depends
only on how many Ritz values coincide, and not on the
Ritz values themselves.
This explains why transverse slices are
possible when we restrict~$\calR$ to belong to a subset of Ritz values
of fixed combinatorial type, such as the generic Ritz values
introduced in Section~\ref{sec:complementary_coordinates}.

\section{Lie theory}
\label{sec:Lie_theory}

For the reader who is not familiar with Lie theory and geometric
terminology, we include an overview of the concepts necessary to
present the results.
Our purpose is not to give formal or abstract definitions, although we
do so as necessary, but to paint a clear picture of the construction
and how it relates to several important areas of mathematics.
This section is independent of the rest of the
paper, and may be ignored by the cognoscenti.

\subsection{Basic geometry: integral curves on a vector field}

This concept is constantly used in
Section~\ref{sec:geometric_definition}, every time a formula contains
an expression like ``$\exp(q\xi)$.''  It also paves the way for our
discussion of integrable systems.

Start with a vector field~$\xi$ on~$M$.
(For concreteness, think of~$M$ as a smooth manifold, although
everything goes through for complex or algebraic varieties, and, in
particular, when we work with~$M=M(n)$ our scalars will be in~$\CC$.)
We have a tangent vector at each point of~$M$.
Starting at some~$x\in M$, we may look for a curve that passes
through~$x$ and is everywhere tangent to the vector field~$\xi$:
\begin{equation}
\label{eq:integral_curve}
\begin{gathered}
\phi(0)=x,\\
\frac{d}{dt}\left(\phi(t)\right)=\xi\left(\phi(t)\right).
\end{gathered}
\end{equation}
By the theory of ordinary differential equations, there exists a
unique solution for all~$t$ in some open interval containing~$0$, but
there is, in general, no reason to expect a solution for all
$t\in\RR$.
\begin{definition}
If \eqref{eq:integral_curve} \emph{does} have a solution for all~$x\in
M$ and all~$t\in\RR$, one says, variously, that the flow defined
by~$\xi$ is complete, or that $\xi$ is complete, or that $\xi$ is
(globally) integrable.
\end{definition}
For obvious reasons, the \emph{flow} defined by~\eqref{eq:integral_curve}
is called~$\exp(t\xi)$, i.e., $\exp(t\xi)(x)=\phi(t)$, where
$\phi(t)$ is defined as above.
One can check that solutions to~\eqref{eq:integral_curve} satisfy
$\exp\bigl((s+t)\xi\bigr)=\exp(s\xi)\exp(t\xi)$.
The word ``flow'' is meant to suggest that as~$t$
changes all the points of~$M$ are smoothly displaced from their
positions, like particles in a fluid.
\begin{example}
Since the vector field~$\xi$ has no singularities, intuitively there
should be no obstruction to integrating it.
If $M$ is compact, then any vector field is complete (since there
exists some~$\veps>0$ such that a solution
to~\eqref{eq:integral_curve} exists for~$|t|<\veps$ for \emph{any} $x\in M$,
and these patch together).
\end{example}
\begin{example}
As an example of what can happen when $M$ is not compact, let $M=\RR$ with
coordinate~$x$, and let $\xi=x^2\,d/dx$.  Then
an integral curve of~$\xi$ through any point would satisfy
$(d/dt)\phi(t)=\phi(t)^2$, whose solutions blow up.
(The previous example shows that nothing bad happens as long as the
integral curve remains bounded.  If $\phi(0)\neq0$, what happens is
that we fall off the ``end'' of~$M$ in a finite amount of time.  The
same phenomenon may be seen with $\xi=d/dx$ and $M=(0,1)$.)
\end{example}
\begin{example}
It may be argued that the last example is misleading: if we consider
$\RR\subseteq\PP^1=\RR\cup\{\infty\}$, then
\[
  \frac{d}{dt}\phi(t)=\phi(t)^2 \gives
  \phi(t)=\frac{1}{\phi(0)^{-1}-t}
\]
becomes a perfectly good flow, with a fixed point at~$x=0$.
No such trick will enable one to integrate~$x^3\,d/dx$, though.
The vector field~$x^3\,d/dx$ has a pole at~$x=\infty$, and there is no
way to embed~$\RR$ as a subset~$U$ of some manifold~$M$ and have a complete
flow on~$M$
whose infinitesimal action on the subset~$U$ is~$x^3\,d/dx$.
\end{example}

\subsection{Classical mechanics and Poisson geometry}

The evolution of a mechanical system can be seen as the orbit of a
point (initial state) under the action of time.
In classical mechanics, the
evolution of the system in time is
determined by Hamilton's equations
\begin{equation}
\label{eq:Hamilton_qp}
\begin{aligned}
\frac{dq_i}{dt} &= \frac{\partial H}{\partial p_i}\\
\frac{dp_i}{dt} &= -\frac{\partial H}{\partial q_i}.
\end{aligned}
\end{equation}
(for a particle in~$\RR^n$), where $q\in\RR^n$ is position and
$p\in\RR^n$ is momentum, and the ``Hamiltonian'' function~$H$ is the
energy of the particle.

We must keep track of time.
Let $f=f(q,p)$ be a classical observable,
which means any smooth function~$\RR^{2n}\to\RR$.
The function~$f$ does not depend on time, but let us write
$f(t)=f\bigl(q(t),p(t)\bigr)\colon \RR^{2n}\to\RR$ for the
observable resulting from picking a point, waiting for time~$t$ to
elapse, and only then measuring the value of~$f$.
(In particular, $f(0)=f=f(q,p)$.)

Instead of considering just position or just momentum, we can re-write
Hamilton's equations as
\begin{equation}
\label{eq:Hamilton}
  \frac{d}{dt}(f(t)) = \{H,f\}(t),
\end{equation}
The right-hand side is the value at time~$t$ of the \emph{Poisson bracket}
\begin{equation}
\label{eq:Poisson_bracket}
  \{H,f\} = \sum_i\frac{\partial H}{\partial p_i}\frac{\partial f}{\partial q_i}-\frac{\partial H}{\partial q_i}\frac{\partial f}{\partial p_i}.
\end{equation}
Physicists have concluded from staring at \eqref{eq:integral_curve}
and \eqref{eq:Hamilton} that the Hamiltonian~$H$ \emph{generates} time
evolution: if the map~$f\mapsto\{H,f\}$ is thought of as defining a
vector field~$\xi_H= \sum_i\frac{\partial H}{\partial
  p_i}\frac{\partial}{\partial q_i}-\frac{\partial H}{\partial
  q_i}\frac{\partial}{\partial p_i}$, then
Hamilton's equations are satisfied if and only if the state of the system follows the flow of~$\xi_H$:
\begin{equation}
\label{eq:time_evolution}
  f(t)=f\circ\exp(t\xi_H).
\end{equation}
Therefore, time and time-evolution naturally appear as soon as one
writes down the Hamiltonian.
\emph{Any} function~$H$ (thought of as ``energy'')
generates a complementary coordinate~$t$ (thought of as ``time''),
such that the system evolves in ``time'' so that energy is conserved.  This
is the essence of Hamiltonian geometry.

This formalism also goes through for any space~$M$ with a Poisson
bracket, not just~$\RR^{2n}$.
The appropriate generalization of~\eqref{eq:Poisson_bracket} is
\begin{definition}
A manifold~$M$ is a \emph{Poisson manifold} if the algebra~$\OO(M)$ of
functions~$M\to\RR$ has a \emph{Poisson structure,} i.e., there is a
bracket
\[
  \{\placeholder,\placeholder\}\colon \OO(M)\tensor\OO(M) \to \OO(M)
\]
making $\OO(M)$ into a Lie algebra
($\{\placeholder,\placeholder\}$ is bilinear, antisymmetric,
and satisfies the Jacobi identity)
and satisfying the Leibniz identity
\[
  \{f,gh\}=\{f,g\}h+g\{f,h\}.
\]
\end{definition}
\noindent Classical mechanics takes place on Poisson manifolds.

\subsection{Integrable systems}

Joseph Liouville concerned himself with characterizing those cases when
explicit solutions to the equations~\eqref{eq:Hamilton_qp} may actually be found.
He proved this is the case when
there are sufficiently many
(independent) commuting Hamiltonians; this is known as complete
integrability in the sense of Liouville.
The idea is that for a completely integrable system we can
(explicitly) find canonical coordinates.

We have seen in~\eqref{eq:time_evolution} that every function (``Hamiltonian'') on a phase space
has a complementary coordinate associated with it, given by following some
flow.  Equation~\ref{eq:Hamilton} says that a function~$f$ is
constant along the flow of~$\xi_H$ if and only if $\{H,f\}=0$.
(In classical mechanics, such functions are called \emph{(first) integrals.})
So, if $f_1,\,\ldots,\,f_n$ are functions such that $\{H,f_i\}=0$ and
$\{f_i,f_j\}=0$ for all $i$~and~$j$, then the trajectory of the system is
contained in a level set of~$(f_1,\,\ldots,\,f_n)$.
(Commutativity with respect to the Poisson bracket means that the flow
corresponding to each function conserves all the other functions.)

Not every system has sufficiently many independent integrals of
motion.  (``Independent'' means that their differentials are linearly
independent (on a dense open subset of the phase space).)
If there are enough independent first integrals which
are simultaneously observable ($\equiv$ commutativity with respect
to~$\{\placeholder,\placeholder\}$), then
if the associated flows
are complete we get a system of
coordinates on the entire phase space
given by the
values of each of the functions together with the dual coordinates
along the level sets (fibres!) given by following the flows.

Integrable systems are commonly defined in the case when~$M$ is a
\emph{symplectic} manifold, rather than in the more general case of a Poisson
manifold.
(A symplectic manifold is a manifold which
has a closed non-degenerate $2$-form; it is a very special kind of
Poisson manifold.)
Since we will later assert that \KandW's Gelfand--Zeitlin
algebra defines an integrable system\footnote{This system, and its
  complete integrability, was already known to Thimm~\cite{T} and Guillemin--Sternberg~\cite{GS} in the 1980s.}
on~$M(n)$, which is not symplectic, we give the more general
definition.

First of all, what is the maximum possible number of independent
commuting Hamiltonians?  (For a symplectic~$M$, this is
$\frac{1}{2}\dim M$.)  Let $M$ be a Poisson manifold.  The \emph{rank}
of the Poisson structure is defined to be the maximum possible number
of linearly independent Hamiltonian vector fields at a point, i.e.,
\[
  \rank\{\placeholder,\placeholder\} = \max_{x\in
    M}\dim\langle\,(\xi_f)_x \mid f\in\OO(M)\,\rangle.
\]
One can show that any Poisson-commutative algebra of functions has
dimension at most~$\dim M-\frac{1}{2}\rank\{\placeholder,\placeholder\}$.
\begin{definition_}
\label{def:integrable_system}
A Poisson manifold~$M$ of rank~$r$ together with a maximal
Poisson-commutative algebra~$A\subseteq\OO(M)$ is a \emph{(completely)
integrable Hamiltonian system} if and only if
\[
  \dim A = \dim M-\frac{r}{2}.
\]
\end{definition_}
Liouville showed~\cite{L} how to construct canonical coordinates for
an integrable system and solve Hamilton's equations.  (A modern
treatment may be found in~\cite{A}; for the generalization to the
Poisson case, see, for example, \cite{LMV}.)

The coordinates dual to the functions are called \emph{angle coordinates.}%
\footnote{Beware that the corresponding \emph{action} coordinates are
  not usually the same as the particular functions used to specify an
  integrable system.  In the case of generic Ritz values, the Ritz
  values themselves will be action coordinates (cf.~\cite[Theorem~5.23]{KW2}), but this point is not
  needed in this paper and we will not pursue it.}
Denote them by~$\varphi_i$; we remark that once the angle coordinates
are known, the system~\eqref{eq:Hamilton_qp} is equivalent to
\begin{align*}
  \frac{d}{dt}f_i(t)&=0,\\
  \frac{d}{dt}\varphi_i(t)&=\text{constant},
\end{align*}
which is trivial to integrate.

Note that the ``canonical'' coordinates do depend on
the choice of Hamiltonians.
Also, the complementary coordinates are measured from a basepoint,
which must be specified (note that the level sets may not even be connected).

\begin{example_}
\label{ex:Lax_pair}
An interesting system is a Lax pair
\[
  \frac{d}{dt}L(t) = [A,L],
\]
where $A$ and $L$ are matrices.  This differential equation describes
an isospectral flow
\[
  L(t)=g(t)L(0){g(t)}^{-1}
\]
($g(t)$ and $A(t)$ are
related via~$dg/dt=Ag$), i.e., the eigenvalues
of~$L$ are invariant over time.
It follows that any functional~$f$ such that $f(gLg^{-1})=f(L)$
is conserved; e.g.,
the coefficients of~$\det\left(\lambda-L(t)\right)$,
or,
alternatively, the functions~$\trace(L^m)$, are conserved
quantities.
Therefore, writing a system as a Lax pair exposes many integrals of
motion
(in fact, \emph{any} completely integrable system can be written as a
Lax pair, although constructing Lax pairs equivalent to integrable
systems, and vice versa, is a far-from-trivial subject).

Not only is this an enormously successful method for actually solving various
integrable systems (including various non-linear partial differential
equations, such as KdV\footnote{$v_\tau=6vv_z+v_{zzz}$}), Lax pairs
bring in Lie theory and geometry in a natural way.  While no
discussion of integrable systems can be complete without mentioning
Lax pairs, the subject is too great to attempt a thorough treatment
here; see~\cite{BBT} for an overview.  For an application
of Lax pairs to the topic of this paper, see~\cite{BP}.
\end{example_}
We adduce that \KandW's theory should be thought of in this
framework.  There the phase space will be the space of all matrices,
and the Hamiltonians---the conserved quantities---will be (certain
symmetric functions of) the Ritz values; the fibres~$M_\calR(n)$ will
be their level sets.  Taking \emph{all} of the Ritz values gives a
\emph{maximal} Poisson-commutative algebra of observables, and their
number is exactly enough to make the system completely integrable.

\section{Kostant--Wallach theory}
\label{sec:geometric_definition}

The algebra of matrices~$M(n)$ has a Poisson structure:
let $\alpha_{ij}$ be the linear functional defined so that
$\alpha_{ij}(x)=x_{ij}$, and let $E_{ij}$ be the matrix with a $1$ in
the $(i,j)$th position and $0$ elsewhere (so
$\alpha_{ij}(E_{kl})=\delta_{ij,\,kl}$ in terms of Kronecker's
$\delta$).  Then
\begin{align*}
  [E_{ij},\,E_{kl}] &= \delta_{jk}E_{il}-\delta_{il}E_{kj}
\intertext{which specifies the Poisson structure as}
  \{ \alpha_{ij},\,\alpha_{kl} \} &=
  \delta_{jk}\alpha_{il}-\delta_{il}\alpha_{kj}.
\end{align*}
This extends naturally to define the Poisson bracket of any two polynomial,
or even holomorphic, functions~$M(n)\to\CC$, because all such may be
written in terms of the~$\alpha_{ij}$.  The Leibniz rule yields
\[
  \{f,g\} = \sum_{ij,\,kl}\{\alpha_{ij},\alpha_{kl}\}\frac{\partial
    f}{\partial\alpha_{ij}}\frac{\partial g}{\partial\alpha_{kl}}.
\]
Keeping this Poisson structure in mind, motivation for
Kostant and Wallach's theory may be found in the theory of completely
integrable Hamiltonian systems.

Kostant and Wallach do not use any general theory in their original
paper; the main actors there are the familiar algebra~$M(n)$ of all
$n\times n$ matrices, and the Lie group~$\GL(n)$ of invertible
$n\times n$ matrices, which acts on~$M(n)$ via the \emph{adjoint
  representation}
\begin{equation}
  \Ad(g)x \defeq gxg^{-1}.
\end{equation}
Recall that we are interested in quantities conserved under the action
of some group, and that we should look in advance for
functions~$M(n)\to\CC$ that Poisson-commute.  Even if we did not know
about Ritz values, we might be led to consider them as follows.

If we were interested in studying $n\times n$ matrices up to
similarity, we would be studying adjoint orbits, in other words,
equivalence classes of matrices under similarity.
A classical problem is to find numerical invariants of these adjoint
orbits, namely all polynomial functions
$f\colon M(n)\to\CC$  (e.g., tr, det) such that
$f(gxg^{-1})=f(x)$ for all $x\in M(n)$ and $g\in\GL(n)$.
The set of such functions may be denoted by
\[
  \Pol\bigl(M(n)\bigr)^{\GL(n)} = \Set{ f\in\Pol\left(M(n)\right) |
    f\text{\ is $\GL(n)$-invariant}}.
\]
The solution to this classical problem is that any such function is
a symmetric polynomial in the roots of the characteristic polynomial,
therefore is
equal to a
polynomial in~$\trace(x^k)$, \ $k=1,\ldots,n$.%
\footnote{As remarked in
  Example~\ref{ex:Lax_pair}, any
  isospectral flow naturally conserves these quantities.  To see that
  any $\Ad$-invariant polynomial is a symmetric function of the roots
  of the characteristic polynomial, a quick way is to observe that any
such polynomial is determined by its value on diagonal matrices.}
Since these functions are constant on adjoint orbits, not only do
they Poisson-commute, but the corresponding vector fields are
identically zero.  However, using this observation, it follows by
induction downwards on~$m$
that $\bigl\{ \trace(x_{m_1})^{k_1},\trace(x_{m_2})^{k_2}\bigr\}=0$ for
any $m_1,\,k_1,\,m_2,\,k_2$ (cf.~\cite[Proposition~2.1]{KW1}).
\begin{example}
Let us write $\bigl\{\trace(x_2),\,\trace(x_3)^2\bigr\}$ out in
coordinates.  We can compute everything in terms of the linear
functionals~$\alpha_{ij}$:
\[
  \bigl\{\trace(x_2),\,\trace(x_3)^2\bigr\} =
  \bigl\{\alpha_{11}+\alpha_{22},\,\alpha_{11}^2+\alpha_{22}^2+\alpha_{33}^2+2\alpha_{12}\alpha_{21}+2\alpha_{13}\alpha_{31}+2\alpha_{23}\alpha_{32}\}.
\]
If we use the Leibniz rule repeatedly and expand, we find that
\begin{multline*}
\{
\alpha_{11}+\alpha_{22},\,\alpha_{11}^2+2\alpha_{12}\alpha_{21}+\cdots\}
=
\{\alpha_{11},\alpha_{11}^2\}+2\{\alpha_{11},\alpha_{12}\alpha_{21}\}+\cdots
=\\=
2\alpha_{11}\{\alpha_{11},\alpha_{11}\}+2\{\alpha_{11},\alpha_{12}\}\alpha_{21}+2\alpha_{12}\{\alpha_{11},\alpha_{21}\}+\cdots
=\\=
0+2\alpha_{12}\alpha_{21}-2\alpha_{12}\alpha_{21}+\cdots
\end{multline*}
then we see that all terms cancel.
\end{example}
So all of the functions Poisson-commute, but
the functions~$\trace(x_m)^k$ for~$m<n$ are not $\Ad$-invariant and
their associated vector fields on~$M(n)$ are non-zero.
Someone looking for a Poisson-commutative algebra of functions
on~$M(n)$ might perhaps stumble upon Ritz values
as a way to greatly enlarge $\Pol\bigl(M(n)\bigr)^{\GL(n)}$.%
\footnote{\KandW\ explicitly mention in the
  abstract of~\cite{KW1} that they regard the
  algebra~$J(n)$ generated by all the $\trace(x_m)^k$ as a classical analogue
  of the \emph{Gelfand--Zeitlin algebra,} which is a commutative (in
  the usual sense) subalgebra of the \emph{universal enveloping
    algebra} of~$M(n)$.}

In any case, \KandW{} begin by considering the algebra
\begin{equation}
\label{eq:classical_Gelfand-Zeitlin}
J(n) = \Pol\bigl(M(1)\bigr)^{\GL(1)}\Pol\bigl(M(2)\bigr)^{\GL(2)}\cdots\Pol\bigl(M(n)\bigr)^{\GL(n)}\subseteq\Pol\bigl(M(n)\bigr),
\end{equation}
which is generated by the functions~$\trace\left((x_m)^k\right)$
for~$m=1,\ldots,n$.\footnote{Taking linear combinations of products of
  functions on the submatrices may be new to many readers.  A point of notation: we are implicitly using the truncation
  map~$x\mapsto x_m$ to embed each~$\Pol\left(M(m)\right)$
  into~$\Pol\left(M(n)\right)$, so each factor
  in~\eqref{eq:classical_Gelfand-Zeitlin} is a subalgebra
  of~$\Pol\bigl(M(n)\bigr)$.  This
  construction of a commutative algebra starting from a system of
  inclusions of subalgebras is associated with Gelfand--Zeitlin
  (a.k.a.~Gelfand--Tsetlin), and is not meant to be intuitively obvious.}
To make this notation clear, let us enumerate
\begin{equation}
\label{eq:enumeration_of_f_j}
  f_1 = \trace(x_1),\quad f_2 = \trace(x_2),\quad f_3 =
\trace(x_2)^2,\quad\text{etc.}
\end{equation}
Then a typical element of~$J(n)$ looks like
\[
  \sum_{\alpha_i\ge0}c_{\alpha_1,\alpha_2,\ldots}f_1^{\alpha_1}f_2^{\alpha_2}\cdots,
\]
which maps $M(n) \to \CC$.  Since the $f_i$ turn out to be
algebraically independent, and they commute, $J(n)$ is isomorphic to a polynomial
algebra~$\CC[f_1,\ldots,f_{\binom{n+1}{2}}]$ in the variables~$f_i$.
\begin{proposition}
[\cite{KW1}, Theorem~0.4]
$J(n)$ is a (maximal) commutative subalgebra
of\/~$\Pol\bigl(M(n)\bigr)$.
Furthermore, for any $f\in J(n)$, the Hamiltonian vector field~$\xi_f$ associated
to~$f$ is globally integrable on~$M(n)$, defining an action of~$\CC$
on~$M(n)$.\footnote{Moreover, this action is given by a nice, explicit
formula; see~\eqref{eq:A-action}.}
\end{proposition}
This richer structure (considering the set of all Ritz values of~$x$
simultaneously) now stands a chance of defining an integrable system
on~$M(n)$.
\begin{proposition}
Let $O_x\subseteq M(n)$ be an adjoint orbit\footnote{We keep coming
  back to adjoint orbits.  Their importance is that $M(n)$ is Poisson
  but not symplectic; the adjoint orbits are the symplectic leaves.}
consisting of regular\footnote{An element $x\in M(n)$ is regular if
  and only if $\dim O_x=n^2-n$.  If $x$ is not regular then
  $\dim O_x$ is strictly lower, and the Hamiltonians $f_i$ have no
  chance of being independent there.  Hence the hypothesis here
  that $x$ be regular.}
elements.  Then the Hamiltonians~$f_i$, \ $1\le i\le \binom{n}{2}$,
form a completely integrable system on~$O_x$.
Moreover, the algebra~$J(n)$
forms an integrable system on~$M(n)$ in the sense of
Definition~\ref{def:integrable_system}.
\end{proposition}
This has a chance of being true because we have exhibited the right
number of commuting Hamiltonians: $\dim O_x=n^2-n$ for regular~$x$,
and the number of Hamiltonians is $1+2+\cdots+(n-1)=n(n-1)\big/2$.
Similarly, the Poisson rank of~$M(n)$ is $n^2-n$,
and the total number of commuting Hamiltonians
including also
$f_{\binom{n}{2}+1},\,\ldots,\,f_{\binom{n+1}{2}}\in\Pol\bigl(M(n)\bigr)^{\GL(n)}$
is $\binom{n+1}{2}=n^2-(n^2-n)\big/2$.
(Things do not work out this nicely if, for instance, we replace
$M(n)$ by the Lie algebra of symplectic matrices; see \cite[Remark~1.7.2]{C}.)

Complete integrability requires
that the Hamiltonians Poisson-commute and that they be
independent.  We have already mentioned the first condition; the
question of independence leads to the notion of \emph{strong
regularity.}
\KandW\ give many equivalent characterizations of
strong regularity.  The ones relevant now are given by their
Theorem~2.7 and the definition immediately preceding it:
\begin{def_thm}
A matrix~$x$ is strongly regular if and only if the
differentials~$(df_i)_x$, \ $1\le i\le\binom{n+1}{2}$, are linearly
independent, if and only if the tangent vectors~$(\xi_{f_i})_x$, \
$1\le i\le\binom{n}{2}$, are linearly independent.
\end{def_thm}
\noindent (The missing vector fields $\xi_{f_{\binom{n}{2}+1}},\ldots$,
corresponding to the elements of~$\bigl(\Pol M(n)\bigr)^{\GL(n)}$, are
zero, as explained before.  Functions whose associated Hamiltonian
vector fields are zero are called \emph{Casimir functions.}  One has
to take them into account when dealing with integrable systems on a
Poisson manifold, rather than the more familiar case of a symplectic manifold.)

In other words, a ``strongly regular'' matrix is a regular point of
the function
$x\mapsto\bigl(f_1(x),\ldots,f_{\binom{n+1}{2}}(x)\bigr)$.  It does
\emph{not} mean what is sometimes called \emph{complete (or strong) regularity,}
that each~$x_m$ be invertible.

\KandW\ prove~\cite[Theorem~2.3]{KW1} that unit upper Hessenberg
matrices are strongly regular, and, therefore, that the set of
strongly regular matrices is a dense open subset of~$M(n)$.  This
proves independence of the ``Hamiltonians''~$f_i$.

We are thus in the situation described before: any matrix will be
described by its Ritz values, which specify a fibre~$M_\calR(n)$,
and the complementary coordinates associated to the Ritz values
(measured from a point on the fibre, which must be specified).
The complementary coordinates will be angle coordinates for
the integrable system.

\medbreak
The Gelfand--Zeitlin group~$A$, central to their theory,
is just the group obtained by integrating the vector fields
corresponding to the functions~$\trace(x_m)^k\in J(n)$ for
$1\le m\le n-1$ and $1\le k\le m$.
This Lie group turns out not to be so mysterious;
it is a a commutative group, isomorphic to~$\CC^{\binom{n}{2}}$.
If~\eqref{eq:enumeration_of_f_j}
are our chosen generators of~$J(n)$, then a typical element of~$A$ can be
written
\begin{equation}
\label{eq:element_of_A}
  a =
  \exp(q_1\xi_{f_1})\exp(q_2\xi_{f_2})\cdots\exp(q_{\binom{n}{2}}\xi_{f_{\binom{n}{2}}})
\end{equation}
where $q_i\in\CC$.
The reader unfamiliar with Lie groups may regard this expression as a
rather formal way of keeping track of the coordinates~$q_i$; the
group multiplication is given by
\[
  \left(\prod_i\exp(q_i\xi_{f_i})\right)\left(\prod_i\exp(q_i'\xi_{f_i})\right)=\prod_i\exp(q_i+q_i')\xi_{f_i}.
\]
The significance of the exponential map in Lie theory is that it
relates maps on Lie groups with maps on Lie algebras.  In this
specific case,
given any matrix~$x\in M(n)$ and any $a\in A$, the matrix~$a\cdot x\in M(n)$
is defined, and it is defined in terms of the vector
fields~$\xi_{f_i}$ (which span a Lie algebra).
This action is computed in~\cite{KW1}, and we will use the results.
In particular, we have the result, in compact form,
\begin{equation}
\label{eq:A-action}
  \exp\left(q\xi_{\trace{(x_m)^k}}\right)\cdot x =
  \Ad\left(\exp\left(-qk(x_m)^{k-1}\right)\oplus\ones\right)x,
\end{equation}
giving an $A$-action on matrices.\footnote{The stated formula is an
  application of~\cite[Theorem~3.3]{KW1}.
  Note the minus sign on the right side of~\eqref{eq:A-action}.}
The key feature is that elements of~$A$ act
by similarity transformations, and that those similarity
transformations involve powers of leading principal submatrices of~$x$.
Moreover,
$\calR(a\cdot x)=\calR(x)$ (``$A$ stabilizes the fibres~$M_\calR(n)$'').

When~$x$ is sufficiently generic, the fibre is a single $A$-orbit,
that is, we have
\[
  M_{\calR(x)}(n) = \Set{ a\cdot x | a \in A }.
\]
The $\binom{n}{2}$ parameters defining $a\in A$,
together with the initial choice of~$x$,
induce a set
of coordinates along the fibre, namely the $q_j$
in~\eqref{eq:element_of_A}.

In general, in view of~\eqref{eq:A-action}, the orbit is given explicitly in terms
of certain subgroups of~$\GL(n)$.  To see which subgroups, note that the
matrix~$\exp\bigl((x_m)^{k-1}\bigr)$ buried on the right-hand side
of~\eqref{eq:A-action} is invertible and is a polynomial in~$x_m$, and
that successively applying~\eqref{eq:A-action} with various
$q$~and~$k$, but keeping $m$ fixed, results
in~$\Ad\bigl(\exp\bigl(p(x_m)\bigr)\oplus\ones\bigr)(x)$, where $p(x_m)$ can be
any polynomial in~$x_m$.
Define
\[
  G_{x,m} = \Set{ g\in\GL(m) | g\text{\ is a polynomial in\ }x_m}.
\]
It will be convenient to consider~$G_{x,m}\subseteq\GL(m)$ as a subgroup of~$\GL(n)$
via the embedding~$g\mapsto\diag(g,\ones)$.
\begin{theorem}
[\cite{KW1}, Theorem~3.7]
The orbit~$A\cdot x$ of (an arbitrary, not necessarily generic)
matrix~$x$ is the image of the mapping
\begin{equation}
  \label{eq:orbit}
  \begin{gathered}
  G_{x,1} \times G_{x,2} \times \cdots \times G_{x,n-1} \to M(n)\\
  \bigl(g(1),\ldots,g(n-1)\bigr) \mapsto \Ad\bigl(g(1)\bigr)\cdots\Ad\bigl(g(n-1)\bigr)(x).
  \end{gathered}
\end{equation}
\end{theorem}
This means that a general element in the $A$-orbit is obtained by
performing a series of similarity transformations of a particular kind.
Therefore, to describe $M_\calR(n)$, we need to understand
how it decomposes into $A$-orbits, and, to describe an $A$-orbit,
we need to understand the kernel of~\eqref{eq:orbit}.

These considerations lead to another characterization of \emph{strong
regularity.}
The condition is that $\dim A\cdot x=\binom{n}{2}$ (the maximum
possible).
\begin{theorem}
\label{thm:strongly_regular_orbit}
[\cite{KW1}, Theorem~3.14]
Let $x$ be strongly regular.
Then the map
\[
G_{x,1}\times\cdots\times G_{x,n-1} \to A\cdot x
\]
is an algebraic isomorphism, so
\[
  A\cdot x \isom G_{x,1}\times\cdots\times G_{x,n-1},
\]
where $G_{x,m}$ is the centralizer of~$x_m$ in\/~$\GL(m)$.
\end{theorem}
This reduces the description of the orbit of any strongly regular matrix
to the description of the groups~$G_{x,m}$, which can be done
explicitly for any matrix.
The only issue left is in picking a set of coordinates for~$G_{x,m}$ that
are somehow natural.  (The group~$A$ was originally defined starting
from particular symmetric functions of the Ritz values, but, for
generic matrices, it will be more convenient to use instead the Ritz
values themselves.)
\begin{example_}
\label{ex:generic_coords}
Let $x_m$ be a regular semi-simple matrix, and suppose $x_m = g_m\Lambda_m g_m^{-1}$,
where $\Lambda_m$ is diagonal.  Then its centralizer is $G_{x,m} = \set{ g_mDg_m^{-1} |
D\text{\ is diagonal with non-zero entries}}\isom(\CC^\times)^m$.
The parameters are the diagonal entries.
\end{example_}
\begin{example_}
\label{ex:Jordan_coords}
Let $x_m$ be any regular matrix, and suppose $x_m = g_mJ_mg_m^{-1}$, where
$J_m$ is in Jordan canonical form.  Then $G_{x,m} = \{g_mD'g_m^{-1}\}$,
where $D'$ is a block-diagonal matrix whose blocks are invertible triangular Toeplitz matrices, one for each Jordan block.  If the Jordan blocks are of sizes~$m_i$ with~$m_1+\cdots+m_t=m$, then $G_{x,m}\isom (\CC^\times)^t\times\CC^{m-t}$.
\end{example_}
In integrable systems language, our
Hamiltonians are independent, but $(f_1,\ldots,f_{\binom{n+1}{2}})$
still has critical points---where the matrix is not strongly regular.
Kostant and Wallach give even more criteria for strong regularity, but
we shall not need them.
The point is that if $x$ is \emph{not} strongly regular,
then describing $A\cdot x$ involves more than just calculating the
groups~$G_{x,m}$ and applying Theorem~\ref{thm:strongly_regular_orbit}, even if each $x_m$ happens to be regular.
\begin{note}
Even when $x$ is strongly regular, if~\ref{G2m} is violated, meaning
$E(x_m)\cap E(x_{m+1}) \neq 0$, then
the description of~$M_{\calR(x)}(n)$ is complicated by the fact that $A$
does not act transitively on the (strongly regular part~$M_{\calR(x)}^\mathrm{sreg}(n)$ of the) fibre, which breaks up into several (isomorphic) orbits.  For example,
\[
\left(%
\begin{array}{ccc}
0 & \hdotsfor{2}\\
1 & 0 & \cdots\\
0 & 1 & \cdots\\
\vdots
\end{array}
\right)
\]
and its transpose belong to distinct orbits.
A point in~~$M_{\calR(x)}^\mathrm{sreg}(n)$ can be
specified by indicating an element of~$A$, the discrete data
needed to specify a particular $A$-orbit, and a point in that orbit.
Then, to define coordinates along the fibre, one needs to pick a representative of
each orbit (note that there is a upper Hessenberg matrix in only one of the orbits); for example, the case~$\calR\equiv 0$ is worked out in~\cite{PS,C}.
\end{note}

When~$x$ is  generic,
a more natural choice of functions
would be the Ritz values~$r_j(x)$ themselves.
These are not globally defined functions, even restricted to the
set of generic matrices (they are defined only on a covering).
However, for any generic~$x$, along with an ordering of each~$E(x_m)$,
it is possible to define vector fields~$\eta_j$, \ $1\le j\le \binom{n}{2}$,
such that the action of~$\CC$ on~$M_{\calR(x)}(n)$ corresponding to the
$j$th Ritz value is given by the action of~$\exp(q\eta_j)$, \ $q\in\CC$.
What follows is the explicit expression of the associated similarities [see the proof of~\cite{KW2}, Theorem~5.5]
\begin{theorem}
\label{thm:KW_5.5}
\begin{gather*}
  \exp(q\eta_j)\cdot x = \Ad\bigl(\gamma_j(e^{-q}))(x),\\
\intertext{where}
  \gamma_j(e^{-q}) = \diag\bigl(g_m\delta_l(e^{-q})g_m^{-1},\,\ones\bigr),
\end{gather*}
where, for~$j=\binom{m}{2}+l$, \ $1\le l\le m$,
\ $g_m\in\GL(m)$ is any matrix such that $x_m=g_m\Lambda_mg_m^{-1}$
and $\delta_l(e^{-q})$ is the $m\times m$ diagonal matrix
\[
  \diag(\underbrace{1,\ldots,1,e^{-q}}_l,1,\ldots,1).
\]
\end{theorem}
Next we put together the similarities associated to all the
eigenvalues of a submatrix~$x_m$:
\begin{corollary}
\label{cor:action}
Let\/ $1\le m\le n-1$ and
\[
  a(m) = \prod_{\binom{m}{2}+1\le j\le\binom{m+1}{2}}\exp(q_j\eta_j), \qquad q_j\in\CC.
\]
Then
\begin{equation}
 \label{eq:formula}
 a(m)\cdot x = \Ad\left(\bigl(g_m\diag(e^{-q_{\binom{m}{2}+1}},\ldots,e^{-q_{\binom{m+1}{2}}})g_m^{-1}\bigr)\oplus\ones\right)x,
\end{equation}
where $g_m\in\GL(m)$ is any matrix such that $x_m=g_m\Lambda_mg_m^{-1}$.
\end{corollary}
\begin{proof}
Apply Theorem~\ref{thm:KW_5.5}, noting that the same $g_m$ works for each~$j=\binom{m}{2}+1,\,\ldots,\,\binom{m+1}{2}$.
\end{proof}
\begin{remark_}
\label{rem:generic_fibre_is_a_torus}
A generic fibre~$M_\calR(n)$ is a single orbit,
and the
element~$\exp(q_1\eta_1)\cdots\exp(q_{\binom{n}{2}}\eta_{\binom{n}{2}})$ of~$A$
acts as the identity on the fibre if and only if each $q_j \in 2\pi i\ZZ$.
(See~\cite{KW2}, Theorem~5.9.)
Corollary~\ref{cor:action} shows that the entries of the diagonal
matrices~$D$ in Example~\ref{ex:generic_coords} are, in fact, the
coordinates~$e^{-q_j}$ dual to the Ritz values.
The condition that the coordinates $\transpose{b_m}$~and~$c_m$ not
vanish is filled automatically here by the exponentials.
Geometrically, a generic fibre is an $\binom{n}{2}$-dimensional \emph{torus,}
because it is isomorphic to a product of $\binom{n}{2}$ copies of the multiplicative
group~$\CC^\times$.
\end{remark_}
\begin{remark}
The coordinates introduced in Example~\ref{ex:Jordan_coords} are a direct
generalization, but it would be interesting to check
whether they satisfy some nice properties analogous to the
generic case.\footnote{In the generic case the ``nice property'' is
  that the diagonal entries in~$D$ in Example~\ref{ex:generic_coords} are
  exponentials of angle coordinates.  This says something about the
  symplectic geometry of generic matrices.  Generalizing this invokes the
  geometry of certain less generic strata of the space of strongly regular matrices,
where the eigenvalues are allowed to coalesce.}
\end{remark}
Finally, Kostant and Wallach define the coordinates~$s_j$
on a generic~$M_\calR(n)$ by picking an initial point.
Recall (Theorem~\ref{thm:upper_Hess}) that $M_\calR(n)$ contains
a unique unit upper Hessenberg matrix~$y$.
Then $s_j$ is defined by
\[
  s_j\left(\exp(q_1\eta_1)\cdots\exp(q_{\binom{n}{2}}\eta_{\binom{n}{2}})\cdot y\right) = e^{-q_j}
\]
(so $s_j(y)=\ones$).

\subsection{Relation to arrow coordinates}
\label{sec:relation_to_arrow_coordinates}

We now prove Claim~\ref{claim:identical_coordinates}, that
$(s_j)$~above are identical to the arrow coordinates~$(\transpose{b_1},\ldots,\transpose{b_{n-1}})$ defined in
Section~\ref{sec:the_complementary_coordinates}, at the end of the matrix
development.
\begin{claim}
If
\[
  x = \mytilde{a}(1)\cdots \mytilde{a}(n-1)\cdot y
\]
with
\[
  \mytilde{a}(m) = \prod_{\binom{m}{2}+1\le j \le \binom{m+1}{2}}\exp(-q_j\eta_j),
\]
then \eqref{eq:arrow} holds with~$\transpose{b_m} = (e^{q_{\binom{m}{2}+1}},\ldots,e^{q_{\binom{m+1}{2}}})$.
\end{claim}
\begin{proof}
Let $g_m \in \GL(m)$, \ $1\le m\le n-1$, be the special matrices defined
by applying the procedure described in Section~\ref{sec:KW_theory}
to the unit Hessenberg matrix~$y$.  We shall need the fact
\begin{equation}
\label{eq:diag_of_Hess}
y_m = g_m\Lambda_mg_m^{-1}
\end{equation}
as well as the normalization
\begin{equation}
  \label{eq:normalization_of_g_m}
  \transpose{e_m}g_m = \ones, \qquad \transpose{e_m}=(0,\ldots,0,1),
\end{equation}
i.e., the last row of~$g_m$ is~$\ones$.

By Corollary~\ref{cor:action}, $x=\Ad\bigl(g(1)\bigr)\cdots\Ad\bigl(g(n-1)\bigr)y$
with
\begin{equation}
\label{eq:g(m)}
  g(m)=g_m\diag(b_m)^{-1}g_m^{-1}\oplus\ones\in G_{y,m}.
\end{equation}

Observe that conjugation by an element of~$G_{y,m}$ leaves $y_m$ fixed,
so
\begin{gather*}
\left(\Ad\bigl(g(n-1)\bigr)y\right)_{n-1} = y_{n-1},\\
\left(\Ad\bigl(g(n-2)\bigr)\Ad\bigl(g(n-1)\bigr)y\right)_{n-2} = 
\Ad\bigl(g(n-2)_{n-2}\bigr)y_{n-2} = y_{n-2},
\end{gather*}
etc., and we have
\begin{gather}
\label{eq:x_m}
  x_m = \Ad\bigl(g(1)_m\cdots g(m-1)_m\bigr)y_m,\\
  x_{m+1} = \Ad\bigl(g(1)_{m+1}\cdots g(m)_{m+1}\bigr)y_{m+1}.
\end{gather}
But \eqref{eq:diag_of_Hess} and \eqref{eq:x_m} imply that
$x_m = Z_m\Lambda_mZ_m^{-1}$ with
\begin{equation}
\label{eq:Z_m}
  Z_m = g(1)_m\cdots g(m-1)_mg_m,
\end{equation}
whence
\begin{equation}
\label{eq:arrow_coords}
\begin{split}
(Z_m^{-1}\oplus 1)x_{m+1}(Z_m\oplus 1)
&= \Ad\bigl((Z_m^{-1}\oplus1)g(1)_{m+1}\cdots g(m)_{m+1}\bigr)y_{m+1}
\\&= \Ad\bigl((g_m^{-1}\oplus 1)g(m)_{m+1}\bigr)y_{m+1}, \quad
\text{by substituting \eqref{eq:Z_m},}
\\&= \Ad\bigl((g_m^{-1}\oplus 1)(g_m\oplus 1)({\diag(b_m)}^{-1}\oplus 1)(g_m^{-1}\oplus1)\bigr)y_{m+1},\quad\text{by~\eqref{eq:g(m)},}
\\&= ({\diag(b_m)}^{-1}\oplus1)(g_m^{-1}\oplus1)y_{m+1}(g_m\oplus1)(\diag(b_m)\oplus1).
\end{split}
\end{equation}
Due to our normalization of~$g_m$, we have (suppressing irrelevant entries)
\[
\left(\begin{matrix}
g_m^{-1} & 0\\
0 & 1
\end{matrix}\right)
y_{m+1}
\left(\begin{matrix}
g_m & 0\\
0 & 1
\end{matrix}\right)
=
\left(\begin{matrix}
g_m^{-1} & 0\\
0 & 1
\end{matrix}\right)
\left(\begin{matrix}
y_m & *\\
\transpose{e_m} & *
\end{matrix}\right)
\left(\begin{matrix}
g_m & 0\\
0 & 1
\end{matrix}\right)
=
\left(\begin{matrix}
g_m^{-1}y_mg_m & *\\
\transpose{e_m}g_m & *
\end{matrix}\right)
=
\left(\begin{matrix}
\Lambda_m & *\\
\ones & *
\end{matrix}\right),
\]
therefore \eqref{eq:arrow_coords} becomes
\[
\begin{pmatrix}
\diag(b_m)^{-1} & 0\\
0 & 1
\end{pmatrix}
\begin{pmatrix}
\Lambda_m & *\\
\ones & *
\end{pmatrix}
\begin{pmatrix}
\diag(b_m) & 0\\
0 & 1
\end{pmatrix}
=
\begin{pmatrix}
\Lambda_m & *\\
\transpose{b_m} & *
\end{pmatrix}.
\]
\end{proof}

\subsection{Coordinates along a non-generic fibre}

As an illustration that the ideas of Section~\ref{sec:geometric_definition}
can be generalized to describe any fibre, we consider a case
studied by Mark Colarusso~\cite{C}.  He suggests looking at
the set of matrices satisfying
\begin{itemize}
\item $x_m$ is regular, $1\le m\le n$.
\item \ref{G2m}, $1\le m\le n-1$.
\end{itemize}
Any such matrix is strongly regular, and,
moreover, the second condition further implies that
$M_{\calR(x)}^{\mathrm{sreg}}(n)$ is a single $A$-orbit.
(This is the largest set of matrices that can be
specified by naming a fibre and an element of~$A$.)

The disadvantage of relaxing~\ref{G1m} is that there will no longer be
a global set of complementary coordinates, because the geometry of the
fibre will vary when the multiplicity of an eigenvalue changes.
A proper generalization of Kostant and Wallach's results in~\cite{KW2}
would consider the geometry of the space of strongly regular matrices
(satisfying \ref{G2m}, for simplicity)
such that the generalized eigenvalues of~$x_m$ have given multiplicity.
However,
since we avoided the technical complication of constructing global
coordinates by considering the fibres individually, we can allow
ourselves to examine a single fibre in this slightly less generic case.

The answer is given by Example~\ref{ex:Jordan_coords}, and the
corresponding arrow coordinates may be computed as follows.
Recall that any eigenvector of~$x_m$ has non-zero last entry.
\begin{claim_}
\label{claim:Jordan}
If the Jordan form of~$x_m$ is $J_{m_1}(\mu_1)\oplus\cdots\oplus J_{m_t}(\mu_t)$,
where each~$J_{m_l}(\mu_l)$ is a lower Jordan block,
then there exists~$g_m\in\GL(m)$ such that $x_m=g_mJ_mg_m^{-1}$
and the last row of~$g_m$ is
\begin{equation}
\label{eq:Jordan_coords}
  (\underbrace{0,0,0,\ldots,1}_{m_1},\underbrace{0,\ldots,1}_{m_2},\ldots).
\end{equation}
\end{claim_}
Given $m$, let $g_m$ be as in Claim~\ref{claim:Jordan}.
Then the arrow
coordinates~$\transpose{b_m}$ are given by the first $m$ entries
of the bottom row of $\diag(g_m^{-1},1)x_{m+1}\diag(g_m,1)$.
The coordinates of a unit upper Hessenberg matrix are \eqref{eq:Jordan_coords}, and
this coincides with the previous construction in case $x$ is a generic matrix.

\bibliography{main}

\end{document}